\documentclass[psamsfonts]{amsart}

\usepackage{amssymb,amsfonts}
\usepackage{enumerate}
\usepackage{mathrsfs}
\usepackage{graphicx}

\usepackage{cite}

\newtheorem{theorem}{Theorem}[section]
\newtheorem{corollary}[theorem]{Corollary}
\newtheorem{proposition}[theorem]{Proposition}
\newtheorem{lemma}[theorem]{Lemma}

\theoremstyle{definition}
\newtheorem{definition}[theorem]{Definition}

\newtheorem{example}[theorem]{Example}
\newtheorem{examples}[theorem]{Examples}

\newtheorem{question}[theorem]{Question}
\newtheorem{questions}[theorem]{Questions}
\newtheorem{remark}[theorem]{Remark}

\newtheorem*{question*}{Question}

\DeclareMathOperator{\Hom}{Hom}

\DeclareMathOperator{\Fid}{Fid}
\DeclareMathOperator{\Id}{Id}

\DeclareMathOperator{\gp}{gp}

\DeclareMathOperator{\Ass}{Ass}
\DeclareMathOperator{\Ann}{Ann}
\DeclareMathOperator{\Spec}{Spec}


\numberwithin{equation}{section}

\begin{document}

\title[On Zero-Divisors of Semimodules and Semialgebras]{On Zero-Divisors of Semimodules and Semialgebras}

\author[Peyman Nasehpour]{\bfseries Peyman Nasehpour}

\address{Peyman Nasehpour\\
Department of Engineering Science \\
Golpayegan University of Technology   \\
Golpayegan\\
Iran}
\email{nasehpour@gut.ac.ir, nasehpour@gmail.com}



\begin{abstract}
In Section \ref{sec:zdm} of the paper, we prove McCoy's property for the zero-divisors of polynomials in semirings. We also investigate zero-divisors of semimodules and prove that under suitable conditions, the monoid semimodule $M[G]$ has very few zero-divisors if and only if the $S$-semimodule $M$ does so. The concept of Auslander semimodules are introduced in this section as well. In Section \ref{sec:mccoysemialgebras}, we introduce Ohm-Rush and McCoy semialgebras and prove some interesting results for prime ideals of monoid semirings. In Section \ref{sec:zdmccoys}, we investigate the set of zero-divisors of McCoy semialgebras. We also introduce strong Krull primes for semirings and investigate their extension in semialgebras.
\end{abstract}

\maketitle

\section{Introduction}\label{sec:intro}

The concept of zero-divisors in ring theory was one of the main concepts that Fraenkel introduced in his paper \cite{Fraenkel1914} in 1914 \cite[p. 59]{Kleiner2007}. Vandiver introduced the concept of zero-divisors in semirings in his 1934 paper \cite{Vandiver1934}, where he introduced the algebraic structure of semirings itself as well \cite{Golan1999}. The main purpose of the current paper is to focus on zero-divisors in semimodules and semialgebras and continue our study of these elements in our 2016 paper \cite{Nasehpour2016}. Before proceeding to explain what we do in this paper, it is important to clarify, from the beginning, what we mean by a semiring in this paper.

In this paper, by a semiring, we understand an algebraic structure, consisting of a nonempty set $S$ with two operations of addition and multiplication such that $(S,+)$ is a commutative monoid with the identity element $0$, $(S,\cdot)$ is a commutative monoid with identity element $1 \not= 0$, multiplication distributes over addition, i.e. $a\cdot (b+c) = a \cdot b + a \cdot c$ for all $a,b,c \in S$ and finally the element $0$ is the absorbing element of the multiplication, i.e. $s \cdot 0=0$ for all $s\in S$. Semimodules over semirings are defined similar to the concept of modules over rings in module theory \cite[Chap. 14]{Golan1999}. For semirings and semimodules and their applications, refer to the books \cite{Golan1999, Golan2003, GondranMinoux2008, HebischWeinert1998}.

Let us recall that for a semiring $S$ and a nonzero $S$-semimodule $M$, an element $s\in S$ is called a zero-divisor on $M$, if there is a nonzero $m\in M$ such that $sm =0$. Section \ref{sec:zdm} of the paper is devoted to the zero-divisors of semimodules. A classical result in commutative algebra states that if $R$ is a commutative ring with a nonzero identity, $f$ is a zero-divisor on $R[X]$, then $f$ can be annihilated by a nonzero constant $b\in R$ \cite[Theorem 2]{McCoy1942}. In Theorem \ref{mccoysemigroupsemimodule}, we show the following:

Let $M$ be an arbitrary $S$-semimodule and $G$ a cancellative torsion-free commutative monoid. If $f\in S[G]$ is a zero-divisor on $M[G]$, then $f$ can be annihilated by a nonzero constant $b\in M$ (McCoy's Theorem for Semimodules).

This useful statement helps us to obtain some interesting results related to the set of zero-divisors $Z_S(M)$ of an $S$-semimodule $M$. In order to explain some of the results that we obtain in Section \ref{sec:zdm}, we need to recall a couple of concepts in semiring theory. A nonempty subset $I$ of a semiring $S$ is said to be an ideal of $S$, if $a+b \in I$ for all $a,b \in I$ and $sa \in I$ for all $s \in S$ and $a \in I$ \cite{Bourne1951}. An ideal $I$ of a semiring $S$ is called a proper ideal of the semiring $S$, if $I \neq S$. A proper ideal $\textbf{p}$ of a semiring $S$ is called a prime ideal of $S$, if $ab\in \textbf{p}$ implies either $a\in \textbf{p}$ or $b\in \textbf{p}$. Similar to commutative algebra, if $M$ is an $S$-semimodule, we define an ideal $\textbf{p}$ of a semiring $S$ an associated prime ideal of $M$, if $\textbf{p} = \Ann(m)$ for some $m\in M$. Note that $\Ann(m)$, for each $m\in M$, is the set of all elements $s\in S$ such that $s\cdot m = 0$. We denote the set of all associated prime ideals of $M$ by $\Ass_S(M)$ or simply $\Ass(M)$ if there is no fear of ambiguity. In Theorem \ref{veryfewzero-divisor4}, we prove that if $Z_S(M) = {\textbf{p}_1}\cup {\textbf{p}_2}\cup \cdots \cup {\textbf{p}_n},$ where ${\textbf{p}_i} \in \Ass_S(M)$ for all $1 \leq i \leq n$, then $$Z_{S[G]}(M[G]) = {\textbf{p}_1}[G]\cup {\textbf{p}_2}[G]\cup \cdots \cup {\textbf{p}_n}[G]$$ and ${\textbf{p}_i}[G] \in \Ass_{R[G]}(M[G])$ for all $1 \leq i \leq n$.

We define an $S$-semimodule $M$ to be primal if $Z_S(M)$ is an ideal of $S$ (see Definition \ref{PrimalsemimoduleDef}). Note that an $S$-semimodule $M$ has Property (A) if each finitely generated ideal $I \subseteq Z_S(M)$ has a nonzero annihilator in $M$ (check Definition \ref{semimodulepropertyA}). In Corollary \ref{PrimalResult}, we show that if $S$ is a weak Gaussian semiring and $G$ is a cancellative torsion-free commutative monoid, then the $S[G]$-semimodule $M[G]$ is primal if and only if the $S$-semimodule $M$ is primal and has Property (A). We recall that a semiring $S$ is weak Gaussian if and only if each of its prime ideals is subtractive \cite[Theorem 19]{Nasehpour2016}. An ideal $I$ of a semiring $S$ is subtractive if $a+b\in I$ and $a\in I$ imply that $b\in $ for all $a,b\in S$ \cite[p. 66]{Golan1999}.  In this section, we also introduce Auslander semimodules. We define an $S$-semimodule $M$ to be an Auslander semimodule, if $Z(S) \subseteq Z_S(M)$ (see Definition \ref{AuslanderSemimodule}). In Theorem \ref{AuslanderSemimoduleThm}, we show that if $M$ is an Auslander $S$-semimodule and has Property (A) and $G$ is a cancellative torsion-free commutative monoid, then $M[G]$ is an Auslander $S[G]$-semimodule.

The definition of Auslander semimodules, inspired by the definition of Auslander modules in \cite{NasehpourAuslander}, is related to Auslander's Zero-Divisor Theorem in commutative algebra, which says that if $R$ is a Noetherian local ring, $M$ is an $R$-module of finite type and finite projective dimension, and $r\in R$ is not a zero-divisor on $M$, then $r$ is not a zero-divisor on $R$ (cf. \cite[p. 8]{Hochster1975} and \cite[Remark 9.4.8]{BrunsHerzog1998}).

At the end of Section \ref{sec:zdm}, we define an $S$-semimodule $M$ to be torsion-free if $Z_S(M) \subseteq Z(S)$ (See Definition \ref{Torsion-free-semimodule}). After that, in Theorem \ref{Torsion-free-Thm1}, we prove that if a semiring $S$ has property (A) and $G$ is a cancellative torsion-free commutative monoid, then the $S[G]$-semimodule $M[G]$ is torsion-free if and only if the $S$-semimodule $M$ is torsion-free.

A semiring $B$ is called to be an $S$-semialgebra, if there is a semiring morphism $\lambda: S \rightarrow B$. For the definition of semiring morphisms, one can refer to \cite[Chap. 9]{Golan1999}.

In Section \ref{sec:mccoysemialgebras} of the paper, we introduce Ohm-Rush and McCoy semialgebras. We define an $S$-semialgebra $B$ to be an Ohm-Rush $S$-semialgebra, if $f\in c(f)B$ for each $f\in B$, where by $c(f)$, we mean the intersection of all ideals $I$ of $S$ such that $f\in IB$ (see Definition \ref{OhmRushSemialgebra}). Note that if $R$ is a commutative ring the term Ohm-Rush algebra has been used for an $R$-algebra that is a content $R$-module by Epstein and Shapiro in \cite{EpsteinShapiro2016}. For more on content modules and algebras, refer to \cite{OhmRush1972}.

Now let $B$ be an $S$-semialgebra. We define $B$ to be a McCoy $S$-semialgebra if $B$ is an Ohm-Rush $S$-semialgebra and $g\cdot f = 0$ for $g,f \in B$ with $g \neq 0$ implies $s\cdot c(f) = 0$ for a nonzero $s\in S$ (check Definition \ref{McCoySemiAlgebraDef}).

After giving these definitions and some examples for them, we generalize some theorems related to weak Gaussian semirings \cite[Definition 18]{Nasehpour2016} and weak content semialgebras \cite[Definition 36]{Nasehpour2016} (for example, see Theorems \ref{WDMcriteria6} and Theorem \ref{NorthcottContentSemialgebra}). These are useful for the main purpose of the last section of the paper. In fact Section \ref{sec:mccoysemialgebras}, beside its interesting results, can be considered as a preparatory section for the last section of the paper which discusses the zero-divisors of McCoy semialgebras.

In Section \ref{sec:zdmccoys} of the present paper, we prove a couple of theorems for the relationship between the set of zero-divisors of a semiring $S$ and a McCoy $S$-semialgebra $B$. For example, in Theorem \ref{veryfewzerodivisorsThm1}, we show that if $S$ is a weak Gaussian semiring, $B$ is a McCoy and weak content $S$-semialgebra with homogeneous $c$, and the corresponding homomorphism $\lambda$ of the $S$-semialgebra $B$ is injective, then $S$ has very few zero-divisors if and only if $B$ does so. Note that in Definition \ref{homogeneouscontentfunction}, we define the content function from an Ohm-Rush $S$-semialgebra $B$ to the set of ideals of $S$ to be homogeneous if $c(s \cdot f) = s \cdot c(f)$ for each $s\in S$ and $f\in B$.

We also show that if $S$ is a weak Gaussian semiring and $B$ is a McCoy and weak content $S$-semialgebra such that the corresponding homomorphism $\lambda$ is injective and $c : B \rightarrow \Fid(S)$ is homogeneous and onto, then $B$ has few zero-divisors if and only if $S$ has few zero-divisors and property (A) (see Theorem \ref{fewzerodivisorsThm1}). Note that by $\Fid(S)$, we mean the set of all finitely generated ideals of $S$.

We finalize the last section by giving the definition of strong Krull prime ideals for semirings. In Definition \ref{strongKrullprimesemiring}, we define a prime ideal $\textbf{p}$ of a semiring $S$ to be a strong Krull prime of $S$, if for any finitely generated ideal $I$ of $S$, there exists a $z\in S$ such that $I\subseteq \Ann(z) \subseteq \textbf{p}$, whenever $I\subseteq \textbf{p}$. After giving this definition, in Lemma \ref{sKprime1}, we prove that if $B$ is a weak content $S$-semialgebra such that the corresponding homomorphism $\lambda$ is injective and $c$ is homogeneous and $\textbf{p}$ is a strong Krull prime of $S$, then either $\textbf{p}B = B$ or $\textbf{p}B$ is a strong Krull prime of $B$. 

Finally, in Theorem \ref{sKprime2}, we show that if $B$ is a McCoy and weak content $S$-semialgebra such that the corresponding homomorphism $\lambda$ is injective, $c$ is homogeneous, and $Z(S)$ is a finite union of strong Krull primes of $S$, then $Z(B)$ is a finite union of strong Krull primes of $B$. We emphasize that one of the corollaries of these results (see Corollary \ref{strongKrullcontent}) is that if $R$ is a commutative ring with a nonzero identity and $B$ is a content $R$-algebra and $Z(R)$ is a finite union of strong Krull primes of $R$, then $Z(B)$ is a finite union of strong Krull primes of $B$. Now we pass to the next section to investigate zero-divisors of semimodules.

\section{Zero-Divisors of Semimodules}\label{sec:zdm}

Let us recall that for a semiring $S$ and a nonzero $S$-semimodule $M$, an element $s\in S$ is called a zero-divisor on $M$, if there is a nonzero $m\in M$ such that $sm =0$. The set of all zero-divisors on $M$ is denoted by $Z_S(M)$, or simply by $Z(M)$, whenever there is no fear of ambiguity.

From monoid theory, we know that a cancellative torsion-free commutative monoid can be embedded into a totally ordered Abelian group (cf. \cite[Corollary 15.7]{Gilmer1972}). We use this to prove Theorem \ref{mccoysemigroupsemimodule}, which is a generalization of a classical result for zero-divisors of polynomial rings showed by Neal Henry McCoy \cite[Theorem 2]{McCoy1942}.

Also, let us recall that any nonzero finitely generated torsion-free Abelian group is isomorphic to the sum of $n$ copies of $\mathbb Z$, i.e., $\oplus_{i=1}^n \mathbb Z$ \cite[Theorem 25.23]{Spindler1994}. This means that arguments on monoid semiring $S[G]$ (semimodule $M[G]$) - where $S$ ($M$) is a semiring (semimodule) and $G$ is a cancellative torsion-free commutative monoid - can be reduced to arguments on finite-variable Laurent polynomial semiring (semimodule) over $S$ ($M$), as for instance, we will see in the proof of Theorem \ref{WDMcriteria5}.

\begin{theorem}[McCoy's Theorem for Semimodules]

\label{mccoysemigroupsemimodule}

Let $M$ be an arbitrary $S$-semimodule and $G$ a cancellative torsion-free commutative monoid. If $f\in S[G]$ is a zero-divisor on $M[G]$, then it can be annihilated by a nonzero constant $b\in M$.

\begin{proof}
For the proof, we take $f\in S[G]$ to be of the form $f = a_0 X^{u_n} + \cdots + a_n X^{u_0}$ ($n \geq 0$, $a_i \in S$, $u_i > u_{i-1}$ and $a_0 \neq 0$) and let $g\in M[G]-\{0\}$ with $f \cdot g =0$. If $g$ is a monomial, then the unique nonzero coefficient of $g$ annihilates $f$. So, we can assume that $g = b_0 X^{t_m} + \cdots + b_m X^{t_0}$ ($m \geq 1$, $b_i \in M$, $t_i > t_{i-1}$ and $b_0 \neq 0$). Now consider the elements $a_0 \cdot g, \ldots, a_n \cdot g$ of $M[G]$. If all of these elements are zero, then $b_0$ annihilates any coefficient of $f$ and the theorem is proved. Otherwise, there is an $r\in \{0,1, \ldots, n\}$ such that $ a_0 \cdot g = \cdots = a_{r-1} \cdot g = 0$, while $a_r \cdot g \neq 0$. This causes $(a_r X^{u_{n-r}} + \cdots + a_n X^{u_0}) \cdot g = 0$, since $f \cdot g = 0$. Clearly, this implies that $ a_r \cdot b_0 = 0$. Therefore, $h = a_r \cdot g \neq 0$ has less monomials than $g$ and $f \cdot h = f \cdot a_r \cdot g = 0$. Therefore, by mathematical induction on the numbers of the monomials of $g$, there is a nonzero constant $b\in M$ such that $f \cdot b = 0$ and the proof is complete.
\end{proof}

\end{theorem}

\begin{corollary}[McCoy's Theorem for Semirings]

\label{mccoysemigroupsemiring}

Let $S$ be a semiring and $G$ a cancellative torsion-free commutative monoid. If $f\in S[G]$ is a zero-divisor on $S[G]$, then it can be annihilated by a nonzero constant $s\in S$.

\end{corollary}

Let us recall that a nonempty subset $I$ of a semiring $S$ is said to be an ideal of $S$, if $a+b \in I$ for all $a,b \in I$ and $sa \in I$ for all $s \in S$ and $a \in I$ \cite{Bourne1951}. We denote the set of all ideals of $S$ by $\Id(S)$. An ideal $I$ of a semiring $S$ is called a proper ideal of the semiring $S$ if $I \neq S$. An ideal $I$ of a semiring $S$ is said to be subtractive, if $a+b \in I$ and $a \in I$ imply that $b \in I$ for all $a,b \in S$. We say that a semiring $S$ is subtractive if each ideal of the semiring $S$ is subtractive. Finally, we note that a proper ideal $\textbf{p}$ of a semiring $S$ is called a prime ideal of $S$, if $ab\in \textbf{p}$ implies either $a\in \textbf{p}$ or $b\in \textbf{p}$.

Now we proceed to prove a theorem for prime ideals of monoid semirings that is a generalization of a theorem for prime ideals of polynomial semirings due to Susan LaGrassa. In fact, LaGrassa in Theorem 2.6 of her dissertation \cite{LaGrassa1995} shows that if $\textbf{p}$ is an ideal of a semiring $S$ and $X$ is an indeterminate over $S$, then $\textbf{p}[X]$ is a prime ideal of $S[X]$ if and only if $\textbf{p}$ is a subtractive prime ideal of $S$.

\begin{theorem}

\label{WDMcriteria5}

Let $S$ be a semiring and $G$ a cancellative torsion-free commutative monoid. Then $\textbf{p}[G]$ is a prime ideal of $S[G]$ if and only if $\textbf{p}$ is a subtractive and prime ideal of $S$.

\begin{proof}

$(\Rightarrow)$: Suppose that $\textbf{p}[G]$ is a prime ideal of $S[G]$ and $a,b\in S$ such that $ab\in \textbf{p}$. Then either $a\in \textbf{p}[G]$ or $b\in \textbf{p}[G]$. Since $\textbf{p}[G] \cap S = \textbf{p}$, we have already proved that $\textbf{p}$ is a prime ideal of $S$. Now suppose that $a,b\in S$ such that $a+b,a \in \textbf{p}$ and take a nonzero $v \in G$. Set $f=a+bX^v$ and $g=b+(a+b)X^v$. Clearly, $fg= ab+(a^2+ab+b^2)X^v+(ab+b^2)X^{2v}$ and so, $fg\in \textbf{p}[G]$. On the other hand, since $\textbf{p}[G]$ is prime, either $f\in \textbf{p}[G]$ or $g\in \textbf{p}[G]$ and in each case, $b\in \textbf{p}$. So, we have showed that $\textbf{p}$ is subtractive.

$(\Leftarrow)$: Suppose that $\textbf{p}$ is a subtractive prime ideal of $S$ and $fg \in \textbf{p}[G]$ for some $f,g \in S[G]$. Imagine $f\in S[G]$ is of the form $f = a_0 X^{u_n} + \cdots + a_n X^{u_0}$ ($n \geq 0$, $a_i \in S$, and $u_i > u_{i-1}$) and $g\in S[G]$ of the form $g = b_0 X^{t_m} + \cdots + b_m X^{t_0}$ ($m \geq 0$, $b_i \in S$ and $t_i > t_{i-1}$). Since $G$ is a cancellative torsion-free commutative monoid, it can be embedded into a torsion-free Abelian group $\gp(G)$ \cite[p. 50]{BrunsGubeladze2009}. Therefore, the subgroup $G_0$ generated by $u_n, \ldots, u_0, t_m, \ldots, t_0$ is isomorphic to a finite copies of $\mathbb Z$ and this means that $fg\in \textbf{p}[G_0]$, where $S[G_0]$ is isomorphic to a finite-variable Laurent polynomial semiring. Now note that by Theorem 39 in \cite{Nasehpour2016}, $\textbf{p}[G_0]$ is a prime ideal and therefore, either $f\in \textbf{p}[G_0]$ or $g\in \textbf{p}[G_0]$ and this means that either $f\in \textbf{p}[G]$ or $g\in \textbf{p}[G]$ and the proof is complete.
\end{proof}

\end{theorem}

Let us recall that a semiring $E$ is called entire if $ab=0$ implies either $a=0$ or $b=0$ for all $a,b\in E$ \cite[p. 4]{Golan1999}.

\begin{corollary}

\label{entiremccoy}
Let $S$ be a semiring and $G$ a cancellative torsion-free commutative monoid. Then $S$ is an entire semiring if and only if $S[G]$ is an entire semiring. In particular, if $k$ is a semifield and $G$ a cancellative torsion-free commutative monoid, then $k[G]$ is an entire semiring. 
\end{corollary}

\begin{remark}
	
\begin{enumerate}
		
    \item By considering Corollary \ref{entiremccoy}, one may ask if it is possible for $S[G]$ to be entire if $S$ is a semiring, while $G$ is not necessarily a cancellative torsion-free commutative monoid. In the following, we give an affirmative answer to this question:
	
	Let us recall that a semiring $S$ is called zerosumfree if $s_1 +s_2 = 0$ implies that $s_1 = s_2=0$, for all $s_1$ and $s_2$ in $S$. Also, a semiring $S$ is called an information algebra if it is both zerosumfree and entire \cite[p. 4]{Golan1999}. Now if $S$ is an information algebra and $G$ is a commutative monid with at least two elements, then the monoid semiring $S[G]$ is entire and here is the proof: \begin{proof} It is clear that if $S$ is an information algebra and $a$ and $b$ are nonzero elements of $S$, then $ab+s$ is also nonzero, for any $s\in S$. Now if $f$ and $g$ are nonzero elements of $S[G]$, then $f$ and $g$, respectively, have monomials of the form $aX^g$ and $bX^h$, where $a$ and $b$ are both nonzero. Clearly, $(ab+s) X^{g+h}$ is a nonzero monimial of $fg$. Therefore, $fg$ is nonzero and the proof is complete.\end{proof} 
	
	\item If $k$ is a field and $G$ is a torsion-free and non-Abelian group, Kaplansky's zero-divisor conjecture states that the group ring $k[G]$ has no non-trivial zero-divisors (cf. \cite[Problem 6]{Kaplansky1970MAA} and \cite[Chap. 13]{Passman1977}). This longstanding algebra conjecture is related to the linear independence of time-frequency shifts (also known as HRT) conjecture in wavelet theory \cite{HeilSpeegle2015}. 
	
\end{enumerate}
\end{remark}

\begin{corollary}

\label{assprimemonoidsemiring}

Let $M$ be an $S$-semimodule and $G$ a cancellative torsion-free commutative monoid. If $\textbf{p} = \Ann(m)$ is a prime ideal of $S$ for some $m\in M$, then $\textbf{p}[G] = \Ann(m)$ is a prime ideal of $S[G]$.

\begin{proof}
Since $\textbf{p} = \Ann(m)$ is subtractive, by Theorem \ref{WDMcriteria5}, $\textbf{p}[G]$ is a prime ideal of $S[G]$. The proof of $\textbf{p}[G] = \Ann(m)$ is straightforward.
\end{proof}

\end{corollary}

In \cite{Davis1964}, it has been defined that a ring $R$ has few zero-divisors, if $Z(R)$ is a finite union of prime ideals. Rings having very few zero-divisors were investigated in \cite{Nasehpour2010}. Modules having few and very few zero-divisors were defined and investigated in \cite{Nasehpour2008, NasehpourPhD, Nasehpour2011, NasehpourPayrovi2010}. We give the following definition and prove some interesting results for zero-divisors of monoid semimodules. Now let $M$ be an $S$-semimodule. Similar to commutative algebra, we define an ideal $\textbf{p}$ of a semiring $S$ to be an associated prime ideal of $M$ if $\textbf{p}$ is a prime ideal of $S$ and $\textbf{p} = \Ann(m)$ for some $m\in M$. The set of all associated prime ideals of $M$ is denoted by $\Ass_S(M)$ or simply $\Ass(M)$.

\begin{definition}

\label{semimodulevfzd}

We define an $S$-semimodule $M$ to have very few zero-divisors, if $Z_S(M)$ is a finite union of prime ideals in $\Ass_S(M)$.
\end{definition}

In order to give some suitable examples for semimodules having very few zero-divisors, we prove the following theorem that is a semiring version of a theorem in commutative algebra due to I. N. Herstein (1923-1988):

\begin{theorem}

\label{primeherstein}

Let $M$ be a nonzero $S$-semimodule. If $\textbf{p}$ is a maximal element of all ideals of the form $\Ann(m)$ of $S$, where $m$ is a nonzero element of $M$, then $\textbf{p}$ is prime.

\begin{proof}
Take $\textbf{p} = \Ann(m)$ and let $ab\in \Ann(m)$. Let $a\notin \textbf{p}$. So, $am \neq 0$ and $\Ann(am) \supseteq \textbf{p}$. Since $\textbf{p}$ is maximal among ideals of $S$ in the form of $\Ann(x)$ with $x\in M-\{0\}$, we have that $\Ann(am) = \Ann(m)$. On the other hand, $abm = 0$. So, $b\in \Ann(m)$ and this means that $\textbf{p}$ is prime and the proof is complete.
\end{proof}

\end{theorem}

Similar to commutative algebra, an $S$-semimodule $M$ is called Noetherian if any $S$-subsemimodule of $M$ is finitely generated. A semiring $S$ is Noetherian if it is Noetherian as an $S$-semimodule \cite[p. 69]{Golan1999}.

\begin{corollary}

\label{zero-divisorunionofprimes2}

If $S$ is a Noetherian semiring and $M$ is an $S$-semimodule, then $Z(M)$ is a union of subtractive prime ideals of $S$.

\begin{proof}

Note that if $M$ is an $S$-semimodule, then $Z(M) = \bigcup_{m\in M-\{0\}} \Ann(m)$. Set $\mathcal C$ to be the collection of all maximal elements of $\{\Ann(m) : m\in M-\{0\} \}$. Since $S$ is Noetherian, $\mathcal C$ is nonempty. Therefore, $Z(M) = \bigcup_{I\in \mathcal C} I$. By Theorem \ref{primeherstein}, any element of $\mathcal C$ is prime of the form $\Ann(m)$ for some $m\neq 0$, which is a subtractive ideal. Hence, $Z(M)$ is a union of subtractive prime ideals of $S$ and the proof is complete.
\end{proof}

\end{corollary}

\begin{corollary}
	
	\label{Assisnonempty}
	
If $S$ is Noetherian and $M$ is an $S$-semimodule, then $\Ass(M)\neq \emptyset$ and $Z(M)$ is the union of all prime ideals in $\Ass(M)$.
\end{corollary}

Now we prove the following interesting theorem:

\begin{theorem}

\label{veryfewzero-divisor3}

Let $S$ be a Noetherian semiring and $M$ a Noetherian $S$-semimodule. Then $0 < \mid \Ass_S(M) \mid < \infty$.

\begin{proof}

Let $S$ be Noetherian. So by Corollary 	\ref{Assisnonempty}, $\Ass(M)\neq \emptyset$. Assume that $\{\textbf{p}_i = \Ann(m_i)\}_i$ is the family of maximal primes of $Z(M)$. Take $N$ to be the $S$-subsemimodule of $M$ generated by the elements $\{m_i\}_i$, where $\textbf{p}_i = \Ann(m_i)$. Since $M$ is Noetherian, $N$ is generated by a finite number of elements in $\{m_i\}_i$, say by $m_1, m_2, \ldots, m_k$. If any further $m_i$'s exists, it can be written as a linear combination of $m_1, m_2, \ldots, m_k$, i.e., for example we have $m_{k+1} = t_1 m_1 + t_1 m_1 + t_2 m_2 + \cdots + t_k m_k$,  where $t_i \in S$. This implies that $\textbf{p}_1 \cap \textbf{p}_2 \cap \cdots \cap \textbf{p}_k \subseteq \textbf{p}_{k+1}$ and therefore, $\textbf{p}_l \subseteq \textbf{p}_{k+1}$ for some $1 \leq l \leq k$, contradicting the maximality of $\textbf{p}_l$. Hence, there are no further $m_i$'s ($\textbf{p}_i$'s) and the proof is complete.
\end{proof}

\begin{example}
A family of examples for Theorem \ref{veryfewzero-divisor3}: Let $n\geq 2$ be a non-prime natural number. Clearly, the Noetherian semiring $\Id(\mathbb Z_n)$ possesses non-trivial zero-divisors. Now if we take $M = S^n$, then $M$ is also Noetherian. Therefore, by Theorem \ref{veryfewzero-divisor3}, $M$ has very few zero-divisors.
\end{example}

\end{theorem}

The following statement is a semiring version of Remark 2.3.3 in \cite{NasehpourPhD}:

\begin{proposition}

Let $S$ be a semiring and consider the following three conditions on $S$:
	
	\begin{enumerate}
		\item The semiring $S$ is Noetherian.
		\item The semiring $S$ has very few zero-divisors.
		\item The semiring $S$ has few zero-divisors.
	\end{enumerate}
	
	Then $(1) \Rightarrow (2) \Rightarrow (3)$ and none of the implications are reversible.
\end{proposition}

\begin{proof}
For a proof of the implication $(1) \Rightarrow(2) $, use Corollary \ref{Assisnonempty} and Theorem \ref{veryfewzero-divisor3}. It is obvious that (2) implies (3).

$(2) \nRightarrow (1)$: Let $x_1, x_2, x_3, \ldots, x_n, \ldots$ be indeterminates over the semifield $F$ and put $E= F[x_1, x_2, x_3, \ldots, x_n, \ldots]$. Clearly, $E$ is an entire semiring and so, it has very few zero-divisors, while it is not Noetherian.
	
$(3) \nRightarrow (2)$:	Let $\textbf{k}$ be a field and imagine $D=\textbf{k}[x_1, x_2, x_3, \ldots, x_n, \ldots]$. Also set $\mathfrak{m} =(x_1, x_2, x_3, \ldots, x_n, \ldots)$ and $\mathfrak{a}=(x_1^2, x_2^2, x_3^2, \ldots, x_n^2, \ldots)$, where $x_i$s are indeterminates over $k$. Now put $R=D/\mathfrak{a}$. It is easy to check that $R$ is a local ring with the only prime ideal $\mathfrak{m}/\mathfrak{a}$ and $Z(R)=\mathfrak{m}/\mathfrak{a}$, while $\mathfrak{m}/\mathfrak{a}\notin \Ass_R(R)$. In fact, $\Ass_R(R)=\emptyset$.
\end{proof}

In \cite{HuckabaKeller1979}, it has been defined that a ring $R$ has Property (A), if each finitely generated ideal $I \subseteq Z(R)$ has a nonzero annihilator. In \cite{Nasehpour2011}, a generalization of this definition was given for modules. We define semimodules having Property (A) as follows:

\begin{definition}

\label{semimodulepropertyA}

We define an $S$-semimodule $M$ to have Property (A) if each finitely generated ideal $I \subseteq Z_S(M)$ has a nonzero annihilator in $M$.
\end{definition}

The Prime Avoidance Lemma is one of the most important theorems in commutative ring theory \cite[Theorem 81]{Kaplansky1970}. A more general result called Prime Avoidance Lemma for Semirings has been given in \cite[Lemma 3.11]{HetzelLufi2009}. Therefore, we bring the following result without proving it:

\begin{theorem}[Prime Avoidance Theorem for Semirings]
	
	\label{PATsemirings}
	
	Let $S$ be a semiring, $I$ an ideal, and $\textbf{p}_i$ ($1 \leq i \leq n$) subtractive prime ideals of $S$. If $I \subseteq \cup_{i=1}^n \textbf{p}_i$, then $I \subseteq \textbf{p}_i$ for some $i$.
	
	\begin{proof} The proof is exactly the same as the proof given for the Prime Avoidance Theorem in commutative ring theory \cite[Theorem 81]{Kaplansky1970}. Therefore, its proof is omitted here.
	\end{proof}  
	
\end{theorem}

\begin{proposition}
If an $S$-semimodule $M$ has very few zero-divisors, then $M$ has Property (A).

\begin{proof}
Use Theorem \ref{PATsemirings}.
\end{proof}

\end{proposition}

Let $M$ be an $S$-semimodule. We say an element $s\in S$ is $M$-regular if it is not a zero-divisor on $M$, i.e. $s\notin Z_S(M)$. We say an ideal of $S$ is $M$-regular if it contains an $M$-regular element. Note that if $G$ is a commutative monoid and $f =s_1 X^{g_1} + \cdots + s_n X^{g_n}$ is an element of the monoid semiring $S[G]$, then the content of $f$, denoted by $c(f)$, is defined to be the finitely generated ideal $(s_1 , \ldots, s_n)$ of $S$.

\begin{theorem}

\label{PropertyAThm1}

Let $G$ be a cancellative torsion-free commutative monoid and $M$ be an $S$-semimodule. Then the following statements are equivalent:

\begin{enumerate}
 \item The $S$-semimodule $M$ has Property (A).
 \item For all $f \in S[G]$, $f$ is $M[G]$-regular if and only if $c(f)$ is $M$-regular.
\end{enumerate}

\begin{proof}
 $(1) \Rightarrow (2)$: Let the $S$-semimodule $M$ have Property (A). If $f \in S[G]$ is $M[G]$-regular, then $f \cdot m \not= 0$ for all nonzero $m \in M$ and so, $c(f) \cdot m \not= (0)$ for all nonzero $m \in M$ and according to the definition of Property (A), $c(f) \not\subseteq Z_S(M)$. This means that $c(f)$ is $M$-regular.

 Now let $c(f)$ be $M$-regular. So $c(f) \not\subseteq Z_S(M)$ and this means that $c(f)\cdot m \not= (0)$ for all nonzero $m \in M$ and hence $f \cdot m \not= 0$ for all nonzero $m \in M$. Since $G$ is a cancellative torsion-free commutative monoid, by Theorem \ref{mccoysemigroupsemimodule}, $f$ is not a zero-divisor on $M[G]$, i.e. $f$ is $M[G]$-regular.

$(2) \Rightarrow (1)$: Let $I$ be a finitely generated ideal of $S$ such that $I \subseteq Z_S(M)$. Then there exists an $f \in S[G]$ such that $c(f) = I$. But $c(f)$ is not $M$-regular, therefore, according to our assumption, $f$ is not $M[G]$-regular. Therefore, there exists a nonzero $m \in M$ such that $f\cdot m = 0$ and this means that $I\cdot m = (0)$, i.e. $I$ has a nonzero annihilator in $M$ and the proof is complete.
\end{proof}

\end{theorem}

\begin{theorem}

\label{veryfewzero-divisor4}

Let $M$ be an $S$-semimodule and $G$ a cancellative torsion-free commutative monoid. Then the $S[G]$-semimodule $M[G]$ has very few zero-divisors if and only if the $S$-semimodule $M$ has very few zero-divisors.

\begin{proof}
 $(\Leftarrow)$: Let $Z_S(M) = {\textbf{p}_1}\cup {\textbf{p}_2}\cup \cdots \cup {\textbf{p}_n}$, where ${\textbf{p}_i} \in \Ass_S(M)$ for all $1 \leq i \leq n$. First, we show that $Z_{S[G]}(M[G]) = {\textbf{p}_1}[G]\cup {\textbf{p}_2}[G]\cup \cdots \cup {\textbf{p}_n}[G]$. Let $f \in Z_{S[G]}(M[G])$. So, there exists an $m \in M- \lbrace 0 \rbrace $ such that $f\cdot m = 0$ and so, $c(f)\cdot m = (0)$. Therefore, $c(f) \subseteq Z_S(M)$ and this means that $c(f) \subseteq {\textbf{p}_1}\cup {\textbf{p}_2}\cup \cdots \cup {\textbf{p}_n}$ and according to the Prime Avoidance Theorem for Semirings (Theorem \ref{PATsemirings}), we have $c(f) \subseteq {\textbf{p}_i}$, for some $1 \leq i \leq n$ and therefore, $f \in {\textbf{p}_i}[G]$. Now let $f \in {\textbf{p}_1}[G]\cup {\textbf{p}_2}[G]\cup \cdots \cup {\textbf{p}_n}[G]$. Therefore, there exists an $i$ such that $f \in {\textbf{p}_i}[G]$. So, $c(f) \subseteq {\textbf{p}_i}$ and $c(f)$ has a nonzero annihilator in $M$ and this means that $f$ is a zero-divisor on $M[G]$. Note that by Corollary \ref{assprimemonoidsemiring}, ${\textbf{p}_i}[G] \in \Ass_{R[G]}(M[G])$ for all $1 \leq i \leq n$.

$(\Rightarrow)$: Let $Z_{S[G]}(M[G])= \cup _{i=1}^n \textbf{q}_i$, where $\textbf{q}_i \in \Ass_{S[G]}(M[G])$ for all $1\leq i \leq n$. Therefore, $Z_G(M) = \cup _{i=1}^n (\textbf{q}_i \cap S)$. Without loss of generality, we can assume that $\textbf{q}_i \cap S \nsubseteq \textbf{q}_j \cap S$ for all $i \neq j$. Now we prove that $\textbf{q}_i \cap S \in \Ass_S(M)$ for all $1 \leq i \leq n$. Consider $h\in M[G]$ such that $\textbf{q}_i = \Ann (h)$ and $h = m_1X^{g_1}+m_2X^{g_2}+ \cdots+ m_nX^{g_n}$, where $m_1, \ldots,m_n \in M$ and $g_1, \ldots ,g_n \in G$. It is easy to see that $\textbf{q}_i \cap S = \Ann (c(h)) \subseteq \Ann(m_1) \subseteq Z_S(M)$ and by Theorem \ref{PATsemirings}, $\textbf{q}_1 \cap S =\Ann(m_1)$.
\end{proof}
\end{theorem}

Let us recall that a semiring $S$ is called a weak Gaussian semiring, if $c(f)c(g) \subseteq \sqrt {c(fg)} $ for all $f,g \in S[X]$ \cite[Definition 18]{Nasehpour2016}. Also note that a semiring $S$ is weak Gaussian if and only if each of its prime ideals is subtractive \cite[Theorem 19]{Nasehpour2016}. Now let, for the moment, $S$ be a weak Gaussian semiring and $M$ an $S$-semimodule such that the set $Z_S(M)$ of zero-divisors of $M$ is a finite union of prime ideals. One can consider $Z_S(M)= \cup _{i=1}^n \textbf{p}_i$ such that $\textbf{p}_i \nsubseteq \cup _{j=1 \wedge j \neq i}^n \textbf{p}_j$ for all $ 1\leq i \leq n$. Obviously, we have $\textbf{p}_i \nsubseteq \textbf{p}_j$ for all $i \neq j$. Also, it is easy to check that, if $Z_S(M)= \cup _{i=1}^n \textbf{p}_i$ and $Z_S(M)= \cup _{k=1}^m \textbf{q}_k$ such that $\textbf{p}_i \nsubseteq \textbf{p}_j$ for all $i \neq j$ and $\textbf{q}_k \nsubseteq \textbf{q}_l$ for all $k \neq l$, then $m=n$ and $\{\textbf{p}_1, \ldots,\textbf{p}_n\}=\{\textbf{q}_1, \ldots,\textbf{q}_n\}$, i.e., these prime ideals are uniquely determined (For the proof, we have the permission to use Theorem \ref{PATsemirings}, since each prime ideal of a weak Gaussian semiring is subtractive). Now we give the following definition:

\begin{definition}

\label{vfzddefdegreeN}

 An $S$-semimodule $M$ is said to have few zero-divisors of degree $n$, if $Z_S(M)$ is a finite union of $n$ prime ideals $\textbf{p}_1, \ldots,\textbf{p}_n$ of $S$ such that $\textbf{p}_i \nsubseteq \textbf{p}_j$ for all $i \neq j$.
\end{definition}

\begin{theorem}

\label{vfzdThmsemimodule}

Let $S$ be a weak Gaussian semiring, $M$ an $S$-semimodule, and $G$ a cancellative torsion-free commutative monoid. Then the $S[G]$-semimodule $M[G]$ has few zero-divisors of degree $n$ if and only if the $S$-semimodule $M$ has few zero-divisors of degree $n$ and Property (A).

\begin{proof}

$(\Leftarrow)$: By taking into consideration of this assumption that the $S$-semimodule $M$ has Property (A), similar to the proof of Theorem \ref{veryfewzero-divisor4}, if $Z_S(M)= \cup _{i=1}^n \textbf{p}_i$, then $Z_{S[G]}(M[G])= \cup _{i=1}^n \textbf{p}_i[G]$. Also, it is obvious that $\textbf{p}_i[G] \subseteq \textbf{p}_j[G]$ if and only if $\textbf{p}_i \subseteq \textbf{p}_j$, for all $1\leq i,j \leq n$. So, these two imply that the $S[G]$-semimodule $M[G]$ has few zero-divisors of degree $n$.

$(\Rightarrow)$: Note that $Z_S(M) \subseteq Z_{S[G]}(M[G])$. It is easy to check that if $Z_{R[G]}(M[G])= \cup _{i=1}^n \textbf{q}_i$, where $\textbf{q}_i$ is a prime ideal of $S[G]$ for each $1\leq i \leq n$, then $Z_S(M) = \cup _{i=1}^n (\textbf{q}_i \cap S)$. Now we prove that the $S$-semimodule $M$ has Property (A). Let $I \subseteq Z_S(M)$ be a finitely generated ideal of $S$. Choose $f\in S[G]$ such that $I=c(f)$. So, $c(f) \subseteq Z_S(M)$ and obviously, $f \in Z_{S[G]}(M[G])$ and according to Theorem \ref{mccoysemigroupsemimodule}, there exists a nonzero $m\in M$ such that $f\cdot m=0$. This means that $I\cdot m=0$ and $I$ has a nonzero annihilator in $M$. Consider that by a similar discussion in $(\Leftarrow)$, the $S$-semimodule $M$ has few zero-divisors obviously not less than degree $n$ and this completes the proof.
\end{proof}

\end{theorem}

Let us recall that an $R$-module $M$ is said to be primal, if $Z_R(M)$ is an ideal of $R$ \cite{Dauns1997}. Similarly, we define primal semimodules as follows:

\begin{definition}

\label{PrimalsemimoduleDef}

We define an $S$-semimodule $M$ to be primal if $Z_S(M)$ is an ideal of $S$.

\end{definition}

It is easy to check that if $Z_S(M)$ is an ideal of $S$, then it is a prime ideal and therefore, the $S$-semimodule $M$ is primal if and only if $M$ has few zero-divisors of degree one.

\begin{example}

\label{primalexample}

Let $(P,+,0)$ be an idempotent commutative monoid and set $S=P \cup \{1\}$. Now extend addition on $S$ as $a+1=1+a=1$ for all $a\in S$ and define multiplication over $S$ as $ab=0$ for all $a,b \in P$ and $a\cdot 1=1 \cdot a=a$ for all $a\in S$. It is, then, easy to check that $(S,+,\cdot)$ is a semiring and $Z_S(S) = P$ is a prime ideal of $S$ and therefore, $S$ is a primal $S$-semimodule \cite[Proposition 20]{Nasehpour2016}.

\end{example}

\begin{corollary}

\label{PrimalResult}

Let $S$ be a weak Gaussian semiring, $M$ an $S$-semimodule, and $G$ a cancellative torsion-free commutative monoid. Then the $S[G]$-semimodule $M[G]$ is primal if and only if the $S$-semimodule $M$ is primal and has Property (A).
\end{corollary}

Auslander's Zero-Divisor Theorem in module theory states that if $R$ is a Noetherian local ring, $M$ is an $R$-module of finite type and finite projective dimension and $r\in R$ is not a zero-divisor on $M$, then $r$ is not a zero-divisor on $R$ (cf. \cite[p. 8]{Hochster1975}, \cite[Remark 9.4.8]{BrunsHerzog1998}). In \cite{NasehpourAuslander}, we have defined an $R$-module $M$ to be Auslander, if $Z(R) \subseteq Z_R(M)$. This inspires us to give the following definition:

\begin{definition}

\label{AuslanderSemimodule}

We define an $S$-semimodule $M$ to be an Auslander semimodule, if $Z(S) \subseteq Z_S(M)$.
\end{definition}

\begin{examples}
	Here we give some examples for Auslander semimodules:
	\begin{enumerate}
		\item If $S$ is an entire semiring, then obviously, any $R$-semimodule $M$ is Auslander.
		\item If for any nonzero $s\in S$, there is an $x\in M$ such that $s \cdot x \neq 0$, then $\Hom_S(M,M)$ is an Auslander $S$-semimodule, where by $\Hom_S(M,M)$, we mean the set of all $S$-endomorphisms on $M$ \cite[p. 159]{Golan1999}. Here is the proof: \begin{proof} Let $s\in Z_S(S)$. So, there is a nonzero $t\in S$ such that $s \cdot t=0$. Define $f_t : M \longrightarrow M$ by $f_t(x) = t \cdot x$. By assumption, $f_t$ is a nonzero element of $\Hom_R(M,M)$. Clearly, $sf_t = 0$. So, $s\in Z_S(\Hom(M,M))$ and this completes the proof.\end{proof}
		
		\item If $N$ is an $S$-subsemimodule of an $S$-semimodule $M$ and $N$ is Auslander, then clearly, $M$ is also Auslander. Therefore, if $M$ is an Auslander $S$-semimodule, then $M \oplus M^{\prime}$ is also an Auslander $S$-semimodule for any $S$-semimodule $M^{\prime}$. In particular, if $\{M_i\}_{i\in \Lambda}$ is a family of $S$-semimodules and there is an $i\in \Lambda$, say $i_0$, such that $M_{i_0}$ is an Auslander $S$-semimodule, then $\bigoplus_{i\in \Lambda} M_i$ and $\prod_{i\in \Lambda} M_i$ are both Auslander $S$-semimodules.
	\end{enumerate}

\end{examples}

\begin{theorem}

\label{AuslanderSemimoduleThm}

Let $M$ be an Auslander $S$-semimodule, have Property (A) and $G$ a cancellative torsion-free commutative monoid. Then $M[G]$ is an Auslander $S[G]$-semimodule.

\begin{proof}
Let $f\in Z(S[G])$. By Corollary \ref{mccoysemigroupsemiring}, there is a nonzero element $s\in S$ such that $f\cdot s = 0$. This implies that $c(f) \subseteq Z(S)$. But $M$ is an Auslander semimodule, so $Z(S) \subseteq Z_S(M)$, which implies that $c(f) \subseteq Z_S(M)$. On the other hand, $M$ has Property (A). So, $c(f)$ has a nonzero annihilator, which implies that $f\in Z_{S[G]}(M[G])$ and the proof is complete.
\end{proof}

\end{theorem}

Let us recall that if $R$ is a commutative ring with a nonzero identity, $M$ a unital $R$-module, and $Q$ the total ring of fractions of $R$, then $M$ is torsion-free if the natural map $M \rightarrow M \otimes Q$ is injective \cite[p. 19]{BrunsHerzog1998}. It is starightforward to see that $M$ is a torsion-free $R$-module if and only if $Z_R(M) \subseteq Z_R(R).$ Therefore, the notion of Auslander modules defined in \cite{NasehpourAuslander} is a kind of dual to the notion of torsion-free modules. Inspired by this, we define torsion-free semimodules as follows:

\begin{definition}

\label{Torsion-free-semimodule}

We define an $S$-semimodule $M$ to be torsion-free, if $Z_S(M) \subseteq Z(S)$.

\end{definition}

\begin{examples} Here we bring some examples for torsion-free semimodules:
	\begin{enumerate}
		
		\item Clearly, any free $S$-semimodule $F = \bigoplus S$ is torsion-free.
		
		\item Let us recall that if $M$ is an $S$-semimodule, then the content of $m\in M$, denoted by $c(m)$, is defined to be the intersection of all ideals $I$ of $S$ such that $m\in IM$. An $S$-semimdoule $M$ is called to be a content $S$-semimodule if $m \in c(m)M$, for all $m\in M$ \cite[Definition 24]{Nasehpour2016}. Now let $M$ be a content $S$-semimodule such that $c(sm)=sc(m)$, for all $s\in S$ and $m\in M$. Then $M$ is torsion-free and here is its proof: 
		
		\begin{proof} Let $s\in Z(M)$. So, by definition, there is a nonzero $m\in M$ such that $sm=0$. Clearly, this implies that $c(sm)=0$, and so, we have $sc(m)=0$. Note that since $M$ is a content semimodule, $c(m)=0$ if and only if $m=0$. This already implies that the ideal $c(m)$ is nonzero and so, $s\in Z(S)$, Q.E.D.\end{proof}
		
		\item It is straightforward to see that if $P_i$ is a family of torsion-free $S$-semimodules, then the $S$-semimodules $\bigoplus_i P_i$ and $\prod_i P_i$ are also torsion-free.
	\end{enumerate}
\end{examples}

\begin{definition} Let $M$ be an $S$-semimodule. We define the dual semimodule of $M$, denoted by $M^*$, to be the $S$-semimodule $M^* = \Hom_S(M,S)$, where by $\Hom_S(M,S)$, it is meant the set of all semimodule morphisms from $M$ into $S$.
	
	\end{definition} 

\begin{proposition}
	
	Let $M$ be an $S$-semimdoule. If the dual semimodule $M^*$ is nonzero, then it is torsion-free. 
	
	\begin{proof} Clearly, if $s$ is an element of $Z(M^*)$, then there is a nonzero semiring morphism $f\colon M \rightarrow S$ such that $sf = 0$. This means that there is an $m\in M$ such that $f(m) \neq 0$, while $sf(m) = 0$. But $f(m) \in S$. Therefore, $s\in Z(S)$ and the proof is complete.\end{proof} 
	  
\end{proposition}

\begin{theorem}

\label{Torsion-free-Thm1}

Let the semiring $S$ have property (A) and $G$ be a cancellative torsion-free commutative monoid. Then the $S[G]$-semimodule $M[G]$ is torsion-free if and only if the $S$-semimodule $M$ is torsion-free.

\begin{proof}

$(\Rightarrow)$: Let $s\in Z_S(M)$. Clearly, this implies that $s\in Z_{S[G]}(M[G])$. But the $S[G]$-semimodule $M[G]$ is torsion-free. Therefore, $Z_{S[G]}(M[G]) \subseteq Z_{S[G]}(S[G])$. So, $s\in Z_S(S)$.

$(\Leftarrow)$: Let $f\in Z_{S[G]}(M[G])$. By Theorem \ref{mccoysemigroupsemimodule}, there is a nonzero $m\in M$ such that $c(f) \cdot m=0$, which means that $c(f) \subseteq Z_S(M)$. Since $M$ is torsion-free, $c(f) \subseteq Z_S(S)$, and since $S$ has property (A), $f\in Z_{S[G]}(S[G])$ and the proof is complete.
\end{proof}

\end{theorem}

\section{McCoy Semialgebras}\label{sec:mccoysemialgebras}

The main purpose of this section is to introduce McCoy semialgebras. Our definition for McCoy semialgebras is inspired by a classical result in commutative algebra which states that if $R$ is a commutative ring and $f\in R[X]$ is a zero-divisor on $R[X]$, then there is a nonzero $r\in R$ such that $r \cdot f =0$ \cite[Theorem 2]{McCoy1942}. In order to define McCoy semialgebras, we need to recall the definition of content functions and to define Ohm-Rush semialgebras. We recall that an $R$-algebra $B$ is an Ohm-Rush algebra, if $R$ as an $R$-module is a content module \cite[Definition 2.1]{EpsteinShapiro2016}. The concepts of content modules and algebras were introduced and investigated in \cite{OhmRush1972}, \cite{EakinSilver1974}, and \cite{Rush1978}.

\begin{definition}
	
\label{OhmRushSemialgebra}
	
Let $S$ be a semiring and $B$ an $S$-semialgebra.
	
	\begin{enumerate}
		\item The function $c$ from $B$ into ideals of $S$ is called the content function if $c(f)$ is the intersection of all ideals $I$ of $S$ such that $f\in IB$ for each $f\in B$.
		
		\item We define an $S$-semialgebra $B$ to be an Ohm-Rush $S$-semialgebra, if $f\in c(f)B$ for each $f\in B$.
	\end{enumerate}
	
\end{definition}

The proof of the following proposition is straightforward, but we bring it here only for the sake of reference.

\begin{proposition}
	
	If $B$ is an Ohm-Rush $S$-semialgebra, then the following statements hold:
	
	\begin{enumerate}
		
		\item For any $f\in B$, $c(f) = (0) $ if and only if $f = 0$
		
		\item For any $f\in B$, $c(f)$ is a finitely generated ideal of $S$;
		
		\item For any ideal $I$ of $S$, $c(f) \subseteq I$ if and only if $f\in IB$;
		
		\item For each $f,g \in B$, $c(fg) \subseteq c(f) c(g)$.
		
	\end{enumerate}
	
\end{proposition}

Imagine $B$ is an Ohm-Rush $S$-semialgebra and take $f\in B$ and $s\in S$. By definition, $f\in c(f) B$ and this implies that $s\cdot f \in s\cdot c(f) B$. So, $c(s\cdot f) \subseteq s \cdot c(f)$ and therefore, if $s \cdot c(f) = (0)$, then $c(s \cdot f) = (0)$, which implies that $s \cdot f = 0$. On the other hand, we know that if $B$ is an Ohm-Rush $R$-algebra, then $c (r \cdot f) = r \cdot c(f)$ for any $r\in R$ and $f\in B$ if and only if $B$ is a flat $R$-algebra \cite[Corollary 1.6]{OhmRush1972}. Therefore, for flat Ohm-Rush $R$-algebra $B$, $r \cdot f = 0$ implies $ r \cdot c(f) = (0)$ for any $r\in R$ and $f\in B$. So, the question arises if for any $S$-semialgebra $B$, $s \cdot f = 0$ implies $s \cdot c(f) = 0$. The following example shows that this is not the case even for some Ohm-Rush algebras.

\begin{example}

Let $(R,\textbf{m})$ be a discrete valuation ring, where $\textbf{m}$ is its unique maximal ideal. Let $B = R / \textbf{m}^2$. By \cite[Proposition 2.1]{OhmRush1972}, $B$ is an Ohm-Rush $R$-algebra. Now let $g$ and $f$ be in $B-\{0\}$ such that $gf = 0$. If we take $f = r + \textbf{m}^2$ and $g = s + \textbf{m}^2$, we have $r,s \notin \textbf{m}^2$, while $rs \in \textbf{m}^2$ and in the language of valuation rings, we get $0\leq v(r), v(s) \leq 1$, while $v(r) + v(s) = v(rs) \geq 2$. These two imply that $v(r) = v(s) = 1$. Obviously, in the $R$-algebra $B$, we have $sf = 0$, since $s(r+\textbf{m}^2) = rs  + \textbf{m}^2 = \textbf{m}^2$. Now let us calculate $c(f)$:

$c(f) = \bigcap \{I\in \Id(R) : f\in IB\} = \bigcap \{I\in \Id(R) : (r + \textbf{m}^2) \in I \cdot (R/\textbf{m}^2)\} = \bigcap \{I\in \Id(R) : (r + \textbf{m}^2) \in I/\textbf{m}^2 \} = \bigcap \{I\in \Id(R) : r \in I \} = (r)$ and obviously $sc(f) = s(r) = (sr) \neq (0)$, since $r$ and $s$ are nonzero and $R$ is an integral domain.

\end{example}

This argument is the base for our definition for McCoy semialgebras:

\begin{definition}

\label{McCoySemiAlgebraDef}

Let $B$ be an $S$-semialgebra. We define $B$ to be a McCoy $S$-semialgebra, if $B$ is an Ohm-Rush $S$-semialgebra and $g\cdot f = 0$ for $g,f \in B$ with $g \neq 0$ implies $s\cdot c(f) = 0$ for some nonzero $s\in S$.
\end{definition}

\begin{examples}

	\label{McCoyEx}
	
	Here we give some examples for McCoy semialgebras:
	
	\begin{enumerate}
		\item Let $X$ be an indeterminate on a semiring $S$. It is clear that by Corollary \ref{mccoysemigroupsemiring}, $S[X]$ is a McCoy $S$-semialgebra.
		
		\item Let $B$ be a content $S$-semialgebra \cite[Definition 30]{Nasehpour2016}, then by \cite[Proposition 31]{Nasehpour2016}, $B$ is a McCoy $S$-semialgebra.
	\end{enumerate}

\end{examples}

\begin{remark}
	Let us recall that David Fields, a student of Robert Gilmer, has proved that if $R$ is a commutative Noetherian ring and a nonzero power series $f\in R[[X]]$ is a zero-divisor in power series ring $R[[X]]$, then there exists a nonzero element $r\in R$ such that $ r\cdot f = 0$ \cite[Theorem 5]{Fields1971}. We also recall that if $B$ is a content $R$-algebra and a nonzero element $f\in B$ is a zero-divisor in $B$, then there exists a nonzero element $r\in R$ such that $r \cdot f =0$ \cite[Statement 6.1]{OhmRush1972}. In fact, David Field's Theorem for zero-divisors of formal power series is a corollary of this fact that formal power series $R[[X]]$ over a Noetherian ring $R$ is a content $R$-algebra, an important theorem that was recently proved by Epstein and Shapiro \cite[Theorem 2.6]{EpsteinShapiro2016AMS}. On the other hand, if $S$ is a Noetherian semiring,  $X$ is an indeterminate over $S$, and $f=s_0 + s_1 X + s_2 X^2 + \cdots + s_n X^n + \cdots $ is an element of $S[[X]]$, then $c(f)$ is the ideal generated by the coefficients of $f$ and $f\in c(f)S[[X]]$ \cite[Proposition 42]{Nasehpour2016}. Therefore, $S[[X]]$ is an Ohm-Rush $S$-semialgebra. Also, one can easily check that $c(sf)=sc(f)$ for all $s\in S$ and $f\in S[[X]]$ if $S$ is Noetherian. This motivates us to propose the following questions:
\end{remark}

\begin{questions}
	Let $S$ be a Noetherian semiring.
	
	\begin{enumerate}
		
		\item Is $S[[X]]$ a McCoy $S$-semialgebra? 
	
	\item Is $S[[X]]$ is a content $S$-semialgebra?
	
	\end{enumerate}
	
\end{questions}

Now we proceed to give more examples for McCoy semialgebras. First we recall that an $S$-semialgebra $B$ is called a weak content semialgebra if $B$ is an Ohm-Rush semialgebra and the content formula $c(f)c(g) \subseteq \sqrt {c(fg)}$ holds, for all $f,g \in B$ \cite[Definition 36]{Nasehpour2016}. Also, a semiring $S$ is said to be nilpotent-free, if $s^n = 0$ implies $s=0$ for all $s\in S$ and $n\in \mathbb N$. Now we give an interesting family of McCoy semialgebras in the following:

\begin{proposition}
	
	\label{NilMcCoySemialgebra}
	
	Let $S$ be a nilpotent-free semiring and $B$ a weak content $S$-semialgebra. Then $B$ is a McCoy $S$-semialgebra.
	
	\begin{proof}
		Let $gf = 0$, where $g,f \in B$ and $g\neq 0$. Consequently $c(g) c(f) \subseteq \sqrt {c(gf)} = \sqrt {0} = 0$. Now $c(g) \neq 0$, since $g\neq 0$. Take a nonzero element $s\in c(g)$. Obviously, $s\cdot c(f) = (0)$. Q.E.D.
	\end{proof}
	
\end{proposition}

\begin{remark}
	
Let $S$ be a semiring and $B$ be an Ohm-Rush $S$-semialgebra. In Examples \ref{McCoyEx}, we have already mentioned that if $B$ is a content $S$-semialgebra, then it is a McCoy $S$-semialgebra. On the other hand, in Proposition \ref{NilMcCoySemialgebra}, we have shown that if $S$ is nilpotent-free and $B$ is a weak content $S$-semialgebra, then $B$ is a McCoy $S$-semialgebra. 

In the following, we show that there exists a McCoy semialgebra that is not a weak content semialgebra:

Let $S$ be a semiring such that one of its prime ideals is not subtractive. Clearly, by Corollary \ref{mccoysemigroupsemiring}, $S[X]$ is a McCoy $S$-semialgebra, while by Theorem 19 in \cite{Nasehpour2016}, it is not a weak content $S$-semialgebra. For example, consider the idempotent semiring $S = \{ 0,u,1 \}$, where $1+u = u+1 = u$ \cite{LaGrassa1995}. It is clear that the ideal $\{0,u\}$ is prime but not subtractive. Now imagine $f=1+uX$ and $g=u+X$. It is easy to see that $fg = (1+uX)(u+X) = u+uX+uX^2$, $c(fg) = \{0,u\}$ and $c(f)c(g)=S$ while $\sqrt {c(fg)} = \sqrt {\{0,u\}} = \{0,u\}$ and this means that $c(f)c(g) \nsubseteq \sqrt {c(fg)}$, i.e., $S[X]$ is not a weak content $S$-semialgebra, while it is a McCoy $S$-semialgebra. From this discussion, we propose the following question:
	
\end{remark}

\begin{question}
	
	Is there any weak content semialgebra that is not a McCoy semialgebra?
		
\end{question}

Since content semialgebras in general and weak content semialgebras over nilpotent-free semirings are good examples for McCoy semialgebras, we devote the rest of this section to these semialgebras.

Let us note that if $R$ is a commutative ring with a nonzero identity and $G$ a cancellative torsion-free commutative monoid, then $c(f)c(g) \subseteq \sqrt {c(fg)}$, for all $f,g \in R[X]$. Still, there are some semirings that this content formula does not hold \cite[Example 17]{Nasehpour2016}. In fact, a semiring $S$ is called weak Gaussian if $c(f)c(g) \subseteq \sqrt {c(fg)}$, for all $f,g \in R[X]$. And it has been proved in Theorem 19 in \cite{Nasehpour2016} that a semiring $S$ is weak Gaussian if and only if each prime ideal of $S$ is subtractive. Now we give the following theorem which is a generalization of Theorem 19 in \cite{Nasehpour2016}:

\begin{theorem}

\label{WDMcriteria6}

Let $S$ be a semiring and $G$ a cancellative torsion-free commutative monoid. Then the following statements are equivalent:

\begin{enumerate}

\item $c(fg) \subseteq c(f)c(g) \subseteq \sqrt {c(fg)}$, for all $f,g \in S[G]$,
\item $\sqrt I$ is subtractive for each ideal $I$ of the semiring $S$,
\item Each prime ideal $\textbf{p}$ of $S$ is subtractive.

\end{enumerate}

\begin{proof}
$(1) \Rightarrow (2)$: Suppose $I$ is an ideal of $S$ and $a,b \in S$ such that $a+b, a \in \sqrt I$. We need to show that $b\in \sqrt I$. Let $v \in G-\{0\}$ and put $f=a+bX^v$ and $g=b+(a+b)X^v$. Then just like the proof of \cite[Theorem 19]{Nasehpour2016}, we have $b \in \sqrt I$.

$(2) \Rightarrow (3)$: Obvious.

$(3) \Rightarrow (1)$: Let each prime ideal $\textbf{p}$ of $S$ be subtractive. We need to show that $c(f)c(g) \subseteq \sqrt {c(fg)}$, for all $f,g \in S[G]$. Let $f,g \in S[G]$ and suppose that $\textbf{p}$ is a prime ideal of $S$ and $c(fg) \subseteq \textbf{p}$. Obviously, $fg \in \textbf{p}[G]$. Now by Theorem \ref{WDMcriteria5}, the ideal $\textbf{p}[G]$ is a prime ideal of $S[G]$ and so, either $f\in \textbf{p}[G]$ or $g\in \textbf{p}[G]$ and this means that either $c(f) \subseteq \textbf{p}$ or $c(g) \subseteq \textbf{p}$ and in any case $c(f)c(g) \subseteq \textbf{p}$. Consequently, by \cite[Proposition 7.28 (Krull's Theorem)]{Golan1999} - that says that $ \sqrt I = \bigcap_{\textbf{p}\in \Spec_I(S)} \textbf{p}$, where by $\Spec_I(S)$ we mean the set of all prime ideals of $S$ containing $I$, we have $c(f)c(g) \subseteq \bigcap_{\textbf{p}\in \Spec_{c(fg)}(S)} \textbf{p} = \sqrt {c(fg)}$ and the proof is complete.
\end{proof}

\end{theorem}

\begin{corollary}
	
	Let $S$ be a semiring and $G$ a cancellative torsion-free commutative monoid. Then the following statements hold:
	
	\begin{enumerate}
		\item If $S$ is weak Gaussian, then $S[G]$ is a weak content $S$-semialgebra.
		
		\item If $S$ is weak Gaussian and nilpotent-free, then $S[G]$ is a McCoy $S$-semialgebra.
	\end{enumerate}

\end{corollary}

In Theorem 2 of the paper \cite{Northcott1959}, Douglas Geoffrey Northcott (1916--2005) proves that if $R$ is a commutative ring with a nonzero identity and $G$ is a cancellative torsion-free commutative monoid, then $R[G]$ is a content $R$-algebra. In the following, we generalize this great result for subtractive semirings:

\begin{theorem}

\label{NorthcottContentSemialgebra}

Let $S$ be a semiring and $G$ a cancellative torsion-free commutative monoid. Then $S[G]$ is a content $S$-semialgebra if and only if $S$ is a subtractive semiring.

\begin{proof}
($\Rightarrow$): Let $I$ be an ideal of $S$. Take $v\in G-\{0\}$ and $a,b \in S$ such that $a+b, a \in I$. Define $f=1+X^v$ and $g=a+bX^v+aX^{2v}$. It is easy to see that $fg=a+(a+b)X^v+(a+b)X^{2v}+aX^{3v}$, $c(f)=S$, $c(g)=(a,b)$ and $c(fg)=(a+b,a)$. But according to our assumption, Dedekind-Mertens content formula holds and therefore, there exists an $m\in \mathbb N_0$ such that $c(f)^{m+1} c(g) = c(f)^m c(fg)$. This means that $(a,b) = (a+b,a)$, which implies $b\in (a+b,a) \subseteq I$ and $I$ is subtractive.

($\Leftarrow$): Take $f,g \in S[G]$. The same discussion, in the proof of Theorem \ref{WDMcriteria5}, shows that there exists a finitely generated torsion-free Abelian group $G_0$ such that $f,g \in S[G_0]$, which means that $f,g$ can be considered to be elements of a Laurent polynomial semiring with finite number of indeterminates. So, by Theorem 6 in \cite{Nasehpour2016}, Dedekind-Mertens content formula holds for $f,g$ and this finishes the proof.
\end{proof}

\end{theorem}

\begin{remark}

In Theorem 3 in \cite{Nasehpour2011}, it has been shown that if $R$ is a ring and $M$ is a commutative monoid and $R[M]$ is a content $R$-algebra, then $M$ is a cancellative torsion-free monoid. This is not necessarily the case if $R$ is a proper semiring. For example, let $\mathbb B = \{0,1\}$ be the Boolean semifield. Now let $M$ be a monoid with at least two elements. It is, then, easy yo see that the monoid semiring $\mathbb B [M]$ is an entire semiring. Also note that if $f \in \mathbb B [M]$ is nonzero, then $c(f) = \mathbb B$. Therefore, from all we said, we see that if $f,g \in \mathbb B [M]$ are both nonzero, then $c (fg) = \mathbb B = c(f) c (g)$. On the other hand, if either $f=0$ or $g=0$, then $c (fg) = c(f)c(g) = (0)$, which means the $\mathbb B [M]$ is a content (in fact, Gaussian) $\mathbb B$-semialgebra, while $M$ is quite arbitrary.
\end{remark}

We end this section with the following statement for McCoy semialgebras:

\begin{proposition}

Let $B$ be a McCoy $S$-semialgebra. If $S$ is an entire semiring, then so is $B$.

\begin{proof}
Let $g \cdot f = 0$, where $f,g \in B$ and $g \neq 0$. By definition, there exists a nonzero $s\in S$ such that $s\cdot c(f)= 0$. Since $S$ is entire, $c(f) = 0$ and finally $f =0$.
\end{proof}

\end{proposition}

\section{Zero-Divisors of McCoy Semialgebras}\label{sec:zdmccoys}

In this section, we explore some properties of the set of zero-divisors of McCoy semialgebras. We also introduce the concept of strong Krull primes of semirings. We start our investigation with the following lemma:

\begin{lemma}

\label{sub1}

Let $S$ be a semiring and $\textbf{p}_i$ a subtractive prime ideal of $S$ for any $1 \leq i \leq n$. If $B$ is a McCoy $S$-semialgebra, then $Z(S) \subseteq \bigcup^n_{i=1} \textbf{p}_i$ implies that $Z(B) \subseteq \bigcup^n_{i=1} (\textbf{p}_i B)$.

\begin{proof}

Let $f\in Z(B)$. Since $B$ is a McCoy $S$-semialgebra, there is some $s\in S-\{0\}$ such that $s \cdot c(f) = 0$. This means that $c(f) \subseteq Z(S)$. But $Z(S) \subseteq \bigcup^n_{i=1} \textbf{p}_i$, so by Theorem \ref{PATsemirings}, there is an $1 \leq i \leq n$ such that $c(f) \subseteq \textbf{p}_i$ and finally $f\in \textbf{p}_i B$.
\end{proof}

\end{lemma}

Let us recall that if $R$ is a ring and $B$ is an Ohm-Rush $R$-algebra, then $B$ is flat if and only if $c(r \cdot f) = r \cdot c(f)$ for any $r\in R$ and $f\in B$ \cite[Corollary 1.6]{OhmRush1972}. Since the corresponding property for the content function on semialgebras is useful as we will see in this section very soon, we believe it is a good idea to give a name to this property. It is good to mention that the use of the term ``homogeneous" in the following definition stems from the definition of ``homogeneous functions" in mathematical analysis (cf. \cite[Chap. 1, §4.1]{Widder1989}).

\begin{definition}

\label{homogeneouscontentfunction}

Let $B$ be an Ohm-Rush $S$-semialgebra. We say the content function $c$ is homogeneous (of degree 1), if $c(s \cdot f) = s \cdot c(f)$ for each $s\in S$ and $f\in B$.
\end{definition}

\begin{proposition} Let $S$ be a semiring. Then the following statements hold:
	
	\begin{enumerate}
		\item If $B= S[G]$, where $G$ is a commutative monboid, then the content function of the $S$-semialgebra $B$ is homogeneous.
		
		\item If $S$ is Noetherian, then the content function of the $S$-semialgebra $S[[X]]$ is homogeneous.
	\end{enumerate}

\begin{proof}
	
	(1): Let $f=s_1 X^{g_1} + \cdots + s_n X^{g_n}$ be an element of the monoid semiring $S[G]$. Then by using Proposition 23 in \cite{Nasehpour2016}, one can easily see that the content of $f$ is the ideal $(s_1, \dots, s_n)$. Therefore, $c(sf) = sc(f)$, for all $s\in S$.
	
	(2): Let $f=s_0 + s_1 X + \cdots + s_n X^n + \cdots$ be an element of $S[[X]]$ and set $A_f$ to be the ideal generated by the coefficients of $f$. If $S$ is Noetherian, then $A_f = c(f)$ \cite[Proposition 42]{Nasehpour2016}. Therefore, for any $s\in S$, we have the following:
	
	$$c(sf) = A_{sf} = sA_f = sc(f).$$ This finishes the proof.	
	\end{proof}
\end{proposition}

\begin{lemma}

\label{ann}

Let $B$ be an Ohm-Rush $S$-semialgebra with homogeneous $c$ and the corresponding homomorphism $\lambda$ of the $S$-semialgebra $B$ be injective. Then the following statements hold:

\begin{enumerate}

\item $\Ann_S(s)B = \Ann_B(\lambda (s))$ for any $s\in S$;

\item If $Z(B) = \bigcup \textbf{q}_i$, then $Z(S) = \bigcup (\textbf{q}_i \cap S)$.

\end{enumerate}

\begin{proof}
(1): $f \in \Ann_B(\lambda (s)) \Leftrightarrow s\cdot f = 0 \Leftrightarrow s \cdot c(f) = 0 \Leftrightarrow c(f) \subseteq \Ann_S(s) \Leftrightarrow f\in \Ann_S(s)B$.

(2): Let $s\in Z(S)$ be nonzero. So, there is a nonzero $s^{\prime} \in Z(S)$ such that $s\cdot s^{\prime} = 0$. This implies that $\lambda(s) \cdot \lambda(s^{\prime}) = 0$. Since $\lambda$ is injective, $\lambda(s)$ and $\lambda(s^{\prime})$ are both nonzero in $B$. So, $\lambda(s) \in Z(B)$ and therefore, by our assumption, there is an $i$ such that $\lambda(s) \in \textbf{q}_i$, i.e., $s \in \textbf{q}_i \cap S$. Now let $s \in \textbf{q}_i \cap S$. So, $\lambda(s) \in \textbf{q}_i$ and this implies that $\lambda(s)$ is a zero-divisor on $Z(B)$. This means that there is a nonzero $g\in B$ such that $s \cdot g =0$. Since $c$ is homogeneous, $s\cdot c(g) = (0)$ and if we choose a nonzero element $t$ of $c(g)$, we have $s\cdot t = 0$, which means that $s\in Z(S)$ and the proof is complete.
\end{proof}

\end{lemma}

Let us recall that a semiring $S$ has very few zero-divisors if the set of zero-divisors $Z(S)$ of $S$ is a finite union of primes in $\Ass(S)$ \cite[Definition 48]{Nasehpour2016}. For instance, by Theorem 49 in \cite{Nasehpour2016}, any Noetherian semiring has very few zero-divisors. Rings and modules having very few zero-divisors were introduced and studied in \cite{Nasehpour2010}, \cite{NasehpourPayrovi2010} and \cite{Nasehpour2011}.

\begin{theorem}

\label{veryfewzerodivisorsThm1}

Let $S$ be a weak Gaussian semiring and $B$ a McCoy and weak content $S$-semialgebra with homogeneous $c$ and the corresponding homomorphism $\lambda$ of the $S$-semialgebra $B$ be injective. Then $S$ has very few zero-divisors if and only if so does $B$.

\begin{proof}

($ \Rightarrow $): Let $S$ have very few zero-divisors. So by definition, $Z(S)$ is a finite union of primes $\textbf{p}_i$ in $\Ass(S)$. Our claim is that $Z(B)$ is union of the ideals $\textbf{p}_i B$.

Since $\textbf{p}_i = \Ann(s_i)$ is subtractive, by Lemma \ref{sub1}, $Z(B) \subseteq \bigcup^n_{i=1} (\textbf{p}_i B)$. Now let $f\in \textbf{p}_i B$. So $c(f) \subseteq \textbf{p}_i$. But $c$ is homogeneous, so $s_i \cdot f = 0$ and this means that $f \in Z(B)$. Also, note that since $B$ is a weak content $S$-semialgebra, by Lemma \ref{ann}, $\textbf{p}_i B \in \Ass(B)$, in which $B$ has very few zero-divisors.

($ \Leftarrow $): Let $Z(B)= \bigcup _{i=1}^n \textbf{q}_i$, where $\textbf{q}_i \in \Ass(B)$ for all $1\leq i \leq n$. Therefore, by Lemma \ref{ann}, $Z(S) = \bigcup _{i=1}^n (\textbf{q}_i \cap S)$. Without loss of generality, we can assume that $\textbf{q}_i \cap S \nsubseteq \textbf{q}_j \cap S$ for all $i \neq j$. Now we prove that $\textbf{q}_i \cap S\in \Ass(S)$ for all $1 \leq i \leq n$. Consider $f\in B$ such that $\textbf{q}_i = \Ann (f)$ and $c(f)=(s_1,s_2, \ldots,s_m)$. It is easy to see that $\textbf{q}_i \cap S = \Ann (c(f)) \subseteq \Ann(s_1) \subseteq Z(S)$ and since every prime ideal of $S$ is subtractive, by Theorem \ref{PATsemirings}, $\textbf{q}_i \cap S =\Ann(s_1)$.
\end{proof}

\end{theorem}

\begin{corollary}

Let $S$ be a weak Gaussian semiring. Then the following statements hold:

\begin{enumerate}
	
	\item If $X$ is an indeterminate over $S$, then $S$ has very few zero-divisors if and only if $S[X]$ does so.
	
	\item If $X_1, X_2, \ldots, X_n$ are distinct indeterminates over $S$, then $S$ has very few zero-divisors if and only if $S[X_1, X_2, \ldots, X_n]$ does so.

\item If $G$ is a cancellative torsion-free commutative monoid, then $S$ has very few zero-divisors if and only if $S[G]$ does so.

\end{enumerate}

\end{corollary}

\begin{remark}
	
	\label{notcontentEx}
	
	\begin{enumerate}
		
		\item Let $S$ be a semiring. Then it is clear that any content $S$-semialgebra is a weak content and a McCoy $S$-semialgebra. Now we show that there are weak content and McCoy $S$-semialgebras that are not content $S$-semialgebras:
	
	\item Let $S$ be a weak Gaussian and non-subtractive semiring, that is, all its prime ideals are subtractive, while $S$ possesses an ideal that is not subtractive. Note that there are such semirings. For example, refer to Proposition 21 in \cite{Nasehpour2016}. Then the $S$-semialgebra $S[X]$ is a weak content semialgebra, which obviously satisfies McCoy property (See Corollary \ref{mccoysemigroupsemiring} of the current paper), but still it is not a content $S$-semialgebra. So we have already obtained a weak content and a McCoy $S$-semialgebra that is not a content $S$-semialgebra.
	
	\end{enumerate}
\end{remark}

\begin{theorem}

 Let $B$ be a McCoy $S$-semialgebra, which the corresponding homomorphism $\lambda$ is injective. Let the content function $c: B \longrightarrow \Fid(S)$ be homogeneous and onto, where by $\Fid(S)$, we mean the set of finitely generated ideals of $S$. Then the following statements are equivalent:

\begin{enumerate}
 \item $S$ has Property (A),
 \item For all $f \in B$, $f$ is a regular element of $B$ if and only if $c(f)$ is a regular ideal of $S$.
\end{enumerate}

\begin{proof}
 $(1) \Rightarrow (2)$: Let $S$ have Property (A). If $f \in B$ is regular, then for all nonzero $s \in S$, $s\cdot f \not= 0$ and so for all nonzero $s \in S$, $s \cdot c(f) \not= (0)$, i.e. $\Ann(c(f)) = (0)$ and according to the definition of Property (A), $c(f) \not\subseteq Z(S)$. This means that $c(f)$ is a regular ideal of $S$. Now let $c(f)$ be a regular ideal of $S$. So, $c(f) \not\subseteq Z(S)$ and therefore, $\Ann(c(f)) = (0)$. This means that for all nonzero $s \in S$, $s\cdot c(f) \not= (0)$ and so, for all nonzero $s \in S$, $s \cdot f \not= 0$. Since $B$ is a McCoy $S$-semialgebra, $f$ is not a zero-divisor of $B$.

$(2) \Rightarrow (1)$: Let $I$ be a finitely generated ideal of $S$ such that $I \subseteq Z(S)$. Since the content function $c : B \rightarrow \Fid(S)$ is onto, there exists an $f \in B$ such that $c(f) = I$. But $c(f)$ is not a regular ideal of $S$, therefore, according to our assumption, $f$ is not a regular element of $B$. Since $B$ is a McCoy $S$-semialgebra, there exists a nonzero $s \in S$ such that $s \cdot c(f) = 0$ and this means that $s\cdot I = (0)$, i.e. $I$ has a nonzero annihilator and the proof is complete.
\end{proof}

\begin{remark}
In the above theorem the surjectivity condition for the content function $c$ is necessary, because obviously $S$ is a content $S$-semialgebra and the condition (2) is satisfied, while one can choose the semiring $S$ such that it does not have Property (A) (Cf. \cite[Exercise 7, p. 63]{Kaplansky1970}).
\end{remark}

\end{theorem}

\begin{theorem}

\label{fewzerodivisorsThm1}

Let $S$ be a weak Gaussian semiring and $B$ a McCoy and weak content $S$-semialgebra such that the corresponding homomorphism $\lambda$ is injective and $c : B \rightarrow \Fid(S)$ is a homogeneous and an onto function. Then $B$ has few zero-divisors if and only if $S$ has few zero-divisors and property (A).

\end{theorem}

\begin{proof}
($\Leftarrow$): Let $Z(S) = \bigcup \textbf{p}_i$. Take $f\in \textbf{p}_i B$, for some $i$. So, $c(f) \subseteq \textbf{p}_i \subseteq Z(S)$. Since $S$ has property (A), there is some nonzero $s\in S$ such that $s \cdot c(f) = (0)$. Since $c$ is homogeneous and $\lambda$ is injective, $f\in Z(B)$. Now by lemma \ref{sub1}, the proof of this part is complete.

($\Rightarrow$): By considering Lemma \ref{ann}, we only need to prove that $S$ has property (A). Let $I \subseteq Z(S)$ be a finitely generated ideal of $S$. Since $c : B \rightarrow \Fid(S)$ is onto, there is an $f\in B$ such that $c(f) = I$ and by Prime Avoidance Theorem for Semirings, $c(f) \subseteq \textbf{q}_i \cap S$, where $Z(B) = \bigcup \textbf{q}_i$. Now we have $f\in (\textbf{q}_i \cap S)B \subseteq \textbf{q}_i \subseteq Z(B)$. But $B$ is a McCoy $S$-semialgebra. So, there is some nonzero $s\in S$ such that $s\cdot I = (0)$. Q.E.D.
\end{proof}

\begin{corollary}

Let $S$ be a weak Gaussian semiring. Then $S[X_1, X_2, \ldots, X_n]$ has few zero-divisors if and only if $S$ has few zero-divisors and property (A).

\end{corollary}

A prime ideal $\textbf{p}$ of a commutative ring $R$ is said to be a strong Krull prime of $R$, if for any finitely generated ideal $I$ of $R$, there exists a $z\in R$ such that $I\subseteq \Ann(z) \subseteq \textbf{p}$, whenever $I\subseteq \textbf{p}$ \cite{McDowell1975}. For a nice introduction to strong Krull primes and their properties, one can refer to a recent paper by Epstein and Shapiro \cite{EpsteinShapiro2014}.

\begin{definition}

\label{strongKrullprimesemiring}

We define a prime ideal $\textbf{p}$ of a semiring $S$ to be a strong Krull prime of $S$, if for any finitely generated ideal $I$ of $S$, there exists a $z\in S$ such that $I\subseteq \Ann(z) \subseteq \textbf{p}$, whenever $I\subseteq \textbf{p}$.

\end{definition}

\begin{lemma}

\label{sKprime1}

Let $B$ be a weak content $S$-semialgebra such that the corresponding homomorphism $\lambda$ is injective and $c$ is homogeneous. If $\textbf{p}$ is a strong Krull prime of $S$, then either $\textbf{p}B = B$ or $\textbf{p}B$ is a strong Krull prime of $B$.

\begin{proof}
Let $\textbf{p}B \neq B$ and $J$ be a finitely generated ideal of $B$ such that $J \subseteq \textbf{p}B$, where $\textbf{p}$ is a strong Krull prime of $S$. Assume that $J=(f_1, \ldots,f_k)$ for $f_1, \ldots, f_k \in B$. This means that $f_i \in \textbf{p}B$ for all $1 \leq i \leq n$. This implies that $c(f_i) \subseteq \textbf{p}$ for all $1 \leq i \leq n$. Since $\textbf{p}$ is a strong Krull prime of $S$ and $c(f_1) + \cdots + c(f_k) \subseteq \textbf{p}$ is a finitely generated ideal of $S$, there exists an $s \in S$ such that $c(f_1) + \cdots + c(f_k) \subseteq \Ann(s) \subseteq \textbf{p}$. Obviously, $J$ is annihilated by $s$, i.e. $J\subseteq \Ann(\lambda(s))$. The final phase of the proof is to show that $\Ann(\lambda(s)) \subseteq \textbf{p}$. Let $f\in \Ann(\lambda(s))$. So $s \cdot f=0$ and therefore, $s \cdot c(f)=0$. This means that $c(f) \subseteq \Ann(s) \subseteq \textbf{p}$ and finally $f\in \textbf{p}B$.
\end{proof}

\end{lemma}

\begin{theorem}

\label{sKprime2}

Let $B$ be a McCoy and weak content $S$-semialgebra such that the corresponding homomorphism $\lambda$ is injective and $c$ is homogeneous. If $Z(S)$ is a finite union of strong Krull primes of $S$, then $Z(B)$ is a finite union of strong Krull primes of $B$.

\begin{proof}
Let $f\in \textbf{p}_i B$ for some $i$. So, $c(f) \subseteq \textbf{p}_i$. But $\textbf{p}_i$ is a strong Krull prime of $S$ and $c(f)$ is a finitely generated ideal of $S$. So, there exists a $z\in S$ such that $c(f) \subseteq \Ann(z)$. This implies that $f\in Z(B)$. On the other hand, by Lemma \ref{sub1}, $Z(B)= \bigcup^n_{i=1} (\textbf{p}_i B)$ and at least one of those $\textbf{p}_i B$s is a proper ideal of $B$. Therefore, by Lemma \ref{sKprime1}, if $\textbf{p}_i B \neq B$, then $\textbf{p}_i B$ must be a strong Krull prime of $B$ and this completes the proof.
\end{proof}

\end{theorem}

\begin{corollary}
	
\label{sKprime3}
	
Let $S$ be a semiring and $B$ a content $S$-semialgebra such that $c$ is homogeneous. If $Z(S)$ is a finite union of strong Krull primes of $S$, then $Z(B)$ is a finite union of strong Krull primes of $B$.

\begin{proof}
Since any content semialgebra is a McCoy \cite[Proposition 31]{Nasehpour2016} and weak content semialgebra \cite[Proposition 37]{Nasehpour2016}, by Theorem \ref{sKprime2}, the statement holds and the proof is complete.
\end{proof}
\end{corollary}

\begin{corollary}
	
\label{sKprime4}

Let $S$ be a weak Gaussian semiring. If $Z(S)$ is a finite union of strong Krull primes of $S$, then $Z(S[X_1, X_2, \ldots, X_n])$ is a finite union of strong Krull primes of $S[X_1, X_2, \ldots, X_n]$.

\begin{proof}
Since $S$ is a weak Gaussian semiring, each prime ideal of $S$ is subtractive \cite[Theorem 19]{Nasehpour2016}. Therefore, each ideal $\textbf{p}S[X_1, X_2, \ldots, X_n]$ is prime if $\textbf{p}$ is prime (See Theorem \ref{WDMcriteria5}). So, $S[X_1, X_2, \ldots, X_n]$ is a weak content $S$-semialgebra \cite[Proposition 37]{Nasehpour2016}. On the other hand, by Corollary \ref{mccoysemigroupsemiring}, $S[X_1, X_2, \ldots, X_n]$ is a McCoy semialgebra. So, by Theorem \ref{sKprime2}, if $Z(S)$ is a finite union of strong Krull primes of $S$, then $Z(S[X_1, X_2, \ldots, X_n])$ is a finite union of strong Krull primes of $S[X_1, X_2, \ldots, X_n]$.
	\end{proof}
\end{corollary}

\begin{corollary}
	
	\label{strongKrullcontent}
	Let $R$ be a commutative ring with a nonzero identity and $B$ be a content $R$-algebra. If $Z(R)$ is a finite union of strong Krull primes of the ring $R$, then $Z(B)$ is a finite union of strong Krull primes of $B$.
\end{corollary}

\section*{Acknowledgements} This work is supported by Golpayegan University of Technology. Our special thanks go to the Department of Engineering Science in Golpayegan University of Technology for providing all the necessary facilities available to us for successfully conducting this research. We also like to thank Prof. Dara Moazzami, the President of Golpayegan University of Technology, for his help, motivation, and encouragement.

\bibliographystyle{plain}

\end{document}